\documentclass[12pt,a4paper]{article}
\usepackage{amssymb,times,graphicx}
\newtheorem{Proof}{Proof.}
 
 \newenvironment{proof}{\begin{Proof}\rm}{\hfill $\Box$ \end{Proof}}

\newtheorem{lem}{Lemma.}

\newcommand{\GF}{{\mathrm{GF}}}

\newcounter{abbildung}
\begin{document}
\title{ On the axiomatics of projective and affine
geometry\\ in terms of line intersection}
\author{Hans Havlicek\thanks{Corresponding author.} \and Victor Pambuccian}
\date{}

\maketitle \thispagestyle{empty}

\begin{abstract} \noindent By providing explicit definitions, we
show that in both affine and projective geometry of dimension
$\geq 3$, considered as first-order theories axiomatized in terms
of lines as the only variables, and the binary line-intersection
predicate as primitive notion, non-intersection of two lines can
be positively defined in terms of line-intersection.
 \linebreak[4]
Lyndon's preservation theorem.
\end{abstract}

M. Pieri \cite{pie} first noted that projective three-dimensional
space can be axiomatized in terms of lines and line-intersections.
A simplified axiom system was presented in \cite{hi}, and two new
ones in \cite{tro} and \cite{koz}, by authors apparently unaware
of \cite{pie} and \cite{hi}. Another axiom system was presented in
\cite[Ch.\ 7]{sto}, a book devoted to the subject of
three-dimensional projective line geometry.
\par
One of the consequences of \cite{cho} is that axiomatizability in
terms of line-intersections holds not only for $n$-dimension
projective geometry with $n=3$, but for all $n\geq 3$. Two such
axiomatizations were carried out in \cite{pamb}. It follows from
\cite{hav2} that there is \emph{more} than just plain
axiomatizability in terms of line-intersections that can be said
about projective geometry, and it is the purpose of this note to
explore the statements that can be made inside these theories, or
in other words to find the definitional equivalent for the
theorems of Brauner \cite{brau}, Havlicek \cite{hav2}, and
Havlicek \cite{hav1}, which state that \emph{bijective} mappings
between the line sets of projective or affine spaces of the same
dimension $\geq 3$ which map intersecting lines into intersecting
lines stem from collineations, or, for three-dimensional
projective spaces, from correlations. (See also \cite[Ch.\
5]{benz}, \cite{hua}, and \cite{kreuz}).
 \par
We shall also prove that, in the projective case, for $n\geq 4$,
`bijective' can be replaced by `surjective' in the above theorem,
and the same holds in the affine case for $n\geq 3$.
\par
Let ${\cal L}$ denote the one-sorted first-order language, with
individual variables to be interpreted as {\em lines}, containing
as only non-logical symbol the binary relation symbol $\sim$, with
$a\sim b$ to be interpreted as `$a$ intersects $b$' (and thus are
\emph{different} lines).
\par
Given Lyndon's preservation theorem (\cite{lyn}, see also
\cite[Cor.\ 10.3.5, p.\ 500]{hod})---
\par
\textbf{Theorem.} {\em Let ${\cal L}$ be a first order language
containing a sign for an identically false formula, ${\cal T}$ be
a theory in ${\cal L}$, and $\varphi({\bf X})$ be an ${\cal
L}$-formula in the free variables ${\bf X} = (X_1, \ldots, X_n)$.
Then the following assertions are equivalent:
 \vspace{-2.5mm} 
\begin{itemize}\itemsep0mm 
 \item[(i)]
there is a positive  ${\cal L}$-formula $\psi({\bf X})$ such that
${\cal T}\vdash \varphi({\bf X}) \leftrightarrow \psi({\bf X})$;
 \item[(ii)]
for any $\mathfrak{A}, \mathfrak{B}$ $\in Mod({\cal T})$, and each
epimorphism $f:\mathfrak{A}\rightarrow \mathfrak{B}$, the
following condition is satisfied: if ${\bf c}\in \mathfrak{A}^n$
and $\mathfrak{A}\models \varphi({\bf c})$, then
$\mathfrak{B}\models \varphi(f({\bf c}))$.
\end{itemize}
}
\par\noindent
---there must exist a positive ${\cal L}$-definition for
the non-intersection of two lines (note that our `sign for an
identically false formula' is $a\sim a$).

\section{Projective Spaces}
\subsection{Dimension $\geq 4$}

We start with projective geometry of dimension $n \geq 4$. We
shall henceforth write $a\simeq b$ for $a\sim b\vee a=b$, as well
as $(a_1, \ldots, a_p \sim b_1, \ldots, b_q)$ for
$\bigwedge_{1\leq i\leq p, 1\leq j\leq q} a_i\sim b_j$.
\par
We first define the ternary co-punctuality predicate $S$, with
$S(abc)$ standing for `$a, b, c$ are three different lines passing
through the same point' by (addition in the indices, whenever the
stated bound for the index variable is exceeded, is mod 3
throughout the paper)
\begin{equation} \label{1}
S(a_1a_2a_3)\; :\Leftrightarrow\; (\forall\, g)(\exists\, h)\,
g\sim h\wedge\Big( \bigwedge_{i=1}^3(a_i\sim a_{i+1},h)\Big).
\end{equation}
It is easy to see that (\ref{1}) holds when the lines $a_i$ are
different and concurrent. Should the three lines $a_i$ intersect
pairwise in three different points, then they would be coplanar
and, by $n\geq 4$, for a line $g$ which is skew to that plane, we
could not find an appropriate line $h$. Next we define the closely
related ternary predicate $\overline{S}$, where
$\overline{S}(abc)$ stands for `$c$ passes through the
intersection point of $a$ and $b$' by
\begin{equation} \label{2}
\overline{S}(abc)\;:\Leftrightarrow\; S(abc)\vee \Big(a\sim b
\wedge (c=a \vee c=b)\Big),
\end{equation}
and then the quaternary predicate $\#$, with $ab\,\#\, cd$ to be
read as `the intersection point of $a$ and $b$ is different from
that of $c$ and $d$' by
\begin{equation} \label{3}
a_1b_1\,\#\, a_2b_2 \;:\Leftrightarrow\;  (\forall\, g)(\exists\,
h_1h_2) \bigwedge_{i=1}^2 (a_i\sim b_i) \wedge \Bigg(\!\!
\Big(\bigwedge_{i=1}^2 \overline{S}(a_ib_ih_i) \wedge
S(h_1h_2g)\!\Big)\vee
\Big(\bigvee_{i=1}^2\overline{S}(a_ib_ig)\!\Big)\!\!\Bigg).
\end{equation}
\par
In fact, suppose that $P_1:=a_1\cap b_1$ and $P_2:=a_2\cap b_2$
are points and that $g$ is a line. If $P_1$ or $P_2$ is on $g$,
then the existence of $h_1$ and $h_2$ is trivial. If $P_1$ and
$P_2$ are not on $g$ (figure \ref{abb1}), then for $P_1\neq P_2$
there exists a point $H\in g$ which is not on $\langle
P_1,P_2\rangle$, i.~e.\ the line joining $P_1$ and $P_2$; hence
the lines $h_i:=\langle P_i,H\rangle$ ($i=1,2$) have the required
properties.  On the other hand, if $P_1=P_2\notin g$, then
(\ref{3}) cannot be satisfied, since $S(h_1,h_2,g)$ would imply
$h_1\neq h_2$, but $\overline S(a_i,b_i,h_i)$ would force
$h_1=h_2$.
\par
Notice that we can now define positively the negation of line
equality by
\begin{equation}\label{4}
a\neq b\;:\Leftrightarrow\; (\exists\, g)\, ag\,\#\, bg,
\end{equation}
which proves that a surjective map between the sets of lines of
two projective spaces of dimension $n\geq 4$, which maps
intersecting lines into intersecting lines, must be injective as
well.
\par
{\unitlength1.0cm
      \begin{center}
      \begin{minipage}[t]{6.0cm}
         \begin{picture}(6.0,3.5)
         \put(0.05,0.60){$a_1$}
         \put(1.05,0.60){$b_1$}
         \put(2.00,1.60){$h_1$}
         \put(0.40,2.10){$P_1$}
         \put(4.60,0.60){$a_2$}
         \put(5.85,0.60){$b_2$}
         \put(3.80,1.60){$h_2$}
         \put(5.55,1.45){$P_2$}
         \put(3.15,0.70){$H$}
         \put(3.40,3.20){$g$}
         \put(0.0 ,0.0){\includegraphics[width=6.0cm]{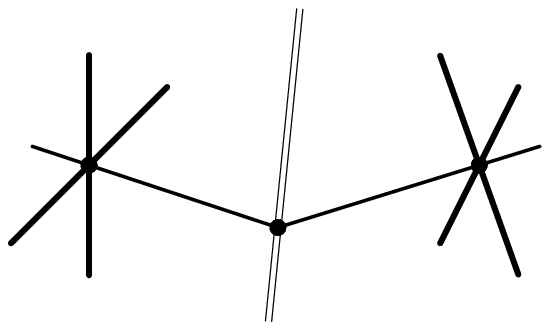}}
      \end{picture}
      {\refstepcounter{abbildung}\label{abb1}
       \centerline{Figure \ref{abb1}.}}
      \end{minipage}
      \hspace{1.5cm}
      \begin{minipage}[t]{6.0cm}
         \begin{picture}(6.0,3.5)
         \put(0.0 ,0.0){\includegraphics[width=6.0cm]{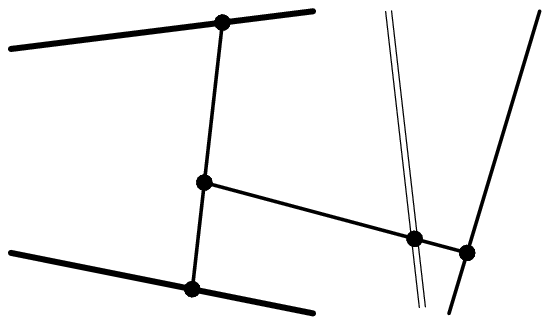}}
         \put(0.20,0.30){$a_1$}
         \put(0.20,2.65){$b_1$}
         \put(2.45,2.40){$b_2$}
         \put(3.45,1.35){$b_3$}
         \put(5.50,1.35){$a_2$}
         \put(4.50,2.80){$g$}
         \end{picture}
         {\refstepcounter{abbildung}\label{abb2}
          \centerline{Figure \ref{abb2}.}}
         \end{minipage}
      \end{center}
}%
\par
We are now ready to define the non-intersection predicate
$\not\sim$ for $n$-dimensional projective spaces with $n\geq 4$.
Let $m=\left[\frac{n-1}{2}\right]$. For $n$ even we have
\begin{eqnarray} \label{5}
a_1\not\sim b_1 &  :\Leftrightarrow  & (a_1=b_1)\vee (\exists\,
a_2\ldots
a_m)(\forall\, g)(\exists\, b_2\ldots b_{m+1})\, \\
& & \bigwedge_{i=2}^{m+1} b_{i}a_{i-1}\,\#\, b_ib_{i-1}\wedge
g\sim b_{m+1}, \nonumber
\end{eqnarray}
and for $n$ odd we have
\begin{eqnarray} \label{6}
a_1\not\sim b_1 &  :\Leftrightarrow  & (a_1=b_1)\vee (\exists\,
a_2\ldots
a_m)(\forall\, g)(\exists\, b_2\ldots b_{m+1}c_2\ldots c_{m+1})\, \\
& & \bigwedge_{i=2}^{m+1} (b_{i}a_{i-1}\,\#\, b_ib_{i-1}\wedge
c_{i}a_{i-1}\,\#\, c_{i}c_{i-1})\wedge b_{m+1}g\,\#\, c_{m+1}g.
 \nonumber
\end{eqnarray}
These two definitions state that, if $a_1$ does not intersect
$b_1$, and if $a_1\neq b_1$, then the set $\{a_1, b_1\}$ can be
extended to a linearly independent set $A:=\{b_1, a_1,\ldots
a_m\}$ (note than if $n=4$, then $m=1$, so there are no $a$'s
bound by the existential quantifier in (\ref{5}) at all) spanning
a subspace $U$ of dimension $2m+1$, i.~e.\ the whole projective
space if $n$ is odd, or a hyperplane if $n$ is even (see
\cite[II.5]{bur}). In both cases, any line $g$ can be reached from
$A$ in the manner described in (\ref{5}) and (\ref{6}), as $g$
lies in $U$ if $n$ is odd, and thus has two different points
common with it, so (\ref{6}) holds, and $g$ intersects $U$ in at
least one point if $n$ is even, so (\ref{5}) holds. See figure
\ref{abb2} for the case $n=6$.

\par
If $a_1$ intersects $a_2$, then the dimension of the subspace
$U$ spanned by any $A$ containing $a_1$ and $a_2$ will be, for $n$
even, at most $n-2$, so there are lines $g$ which do not intersect
$U$, and thus cannot be reached in the manner described in
(\ref{5}), and if $n$ is odd, the dimension of $U$ is at most
$n-1$, so there are lines $g$ which intersects $U$ in one point,
so for those lines definition (\ref{6}), which requires that the
line $g$ intersects $U$ in two different points, cannot hold.
\par
Given (\ref{1}), (\ref{2}), (\ref{3}), it is obvious that
$n$-dimensional projective geometry with $n\geq 3$, can be
axiomatized inside ${\cal L}$, as one can rephrase the axiom
system based on point line incidence of the Veblen-Young type (for
example the one in Lenz \cite[p.\ 19--20]{len} to which lower- and
upper-dimension axioms have been added) in terms of line
intersections only, by replacing each `point $P$' with two
intersecting lines $p_1$ and $p_2$, the equality of two points $P$
and $Q$, which have been replaced by $(p_1, p_2)$  and $(q_1,
q_2)$, by $\overline{S}(p_1p_2q_1)\wedge \overline{S}(p_1p_2q_2)$
and every occurrence of `$P$ is incident with $l$' by
$\overline{S}(p_1p_2l)$. This has been carried out in \cite{pamb}.

\par
Since in some models (e.~g.\ over commutative fields) of
three-dimensional projective geometry there are correlations, $S$
cannot be definable  in terms of $\sim$, so the approach used for
dimensions $\geq 4$ fails in this case. However, $\not\sim$ is
positively definable, with negated equality allowed, in terms of
$\sim$, and it is to this definition that we now turn our
attention.

\subsection{The three-dimensional case}

In the three-dimensional case, we first define the ternary
relation $T$, with $T(abc)$ holding if and only if `either the
three different lines $a,b,c$  intersect pairwise in three
different points (and then we call $abc$ a \emph{tripod}) or they
are concurrent, but do not lie in the same plane (in which case we
call $abc$ a \emph{trilateral})', by
\begin{eqnarray} \label{delta}
T(a_1a_2a_3) & :\Leftrightarrow & (\forall\, g_1g_2)(\exists\,
x_1x_2x_3)\,
(g_1,g_2 \sim x_1,x_2,x_3) \\
& &{}\wedge\Bigg( \bigwedge_{i=1}^3 \Big((x_i \simeq a_i,
a_{i+1})\wedge a_i\sim a_{i+1}\Big)\Bigg) \wedge
\Big(\bigvee_{i=1}^3x_i\neq x_{i+1}\Big). \nonumber
\end{eqnarray}
To see that the above definition holds when $a_1a_2a_3$ is a
trilateral, let $A_i$ be the point of intersection of the lines
$a_i$ and $a_{i+1}$ for $i=1,2,3$ (figure \ref{abb3}). Through
each $A_i$ there is a line $x_i$ intersecting (and different from)
both $g_1$ and $g_2$. The $x_i$ satisfy the conditions of
(\ref{delta}) since they cannot all coincide, given that no single
line can, by the definition of a trilateral, pass through $A_1,
A_2, A_3$. A dual reasoning to that presented for the case in
which $a_1a_2a_3$ is a trilateral shows that the definition
(\ref{delta}) holds for tripods $a_1a_2a_3$ as well.
\par
To see that the only other case that could occur, given that
$a_i\sim a_j$ for all $i\neq j$, namely that in which the three
lines $a_1, a_2, a_3$ are lying in the same plane $\pi$ and have a
point $O$ in common, does not satisfy (\ref{delta}), we choose
$g_1$, $g_2$ such that they are skew, not in $\pi$, and intersect
the line $a_1$ in two points that are different from $O$  (figure
\ref{abb4}). The only line that meets $g_1, g_2$ and two of the
lines $a_1, a_2, a_3$ is $a_1$ itself.
\par
{\unitlength1.0cm
      \begin{center}
      \begin{minipage}[t]{6.0cm}
         \begin{picture}(6.0,3.9)
         \put(0.5,3.7){$g_1$}
         \put(0.5,0.1){$g_2$}
         \put(0.85,3.05){$x_1$}
         \put(2.85,3.05){$x_2$}
         \put(4.80,3.05){$x_3$}
         \put(1.30,1.15){$A_1$}
         \put(2.85,1.95){$A_2$}
         \put(4.35,1.40){$A_3$}
         \put(2.60,1.35){$a_1$}
         \put(1.90,2.15){$a_2$}
         \put(4.00,2.15){$a_3$}
         \put(0.0 ,0.0){\includegraphics[width=6.0cm]{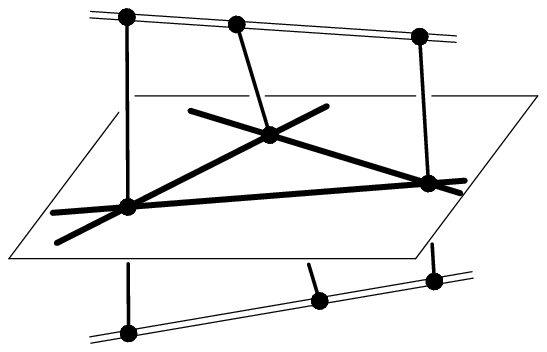}}
      \end{picture}
      {\refstepcounter{abbildung}\label{abb3}
       \centerline{Figure \ref{abb3}.}}
      \end{minipage}
      \hspace{1.5cm}
      \begin{minipage}[t]{6.0cm}
         \begin{picture}(6.0,3.9)
         \put(0.0 ,0.0){\includegraphics[width=6.0cm]{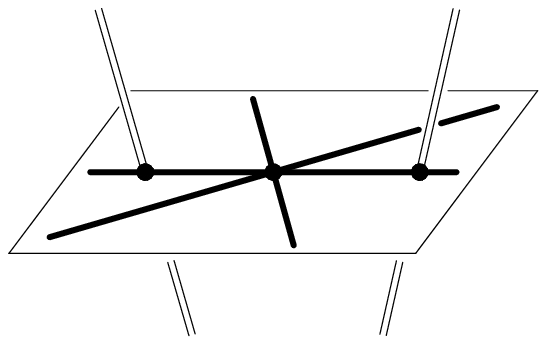}}
         \put(1.25,3.5){$g_1$}
         \put(4.55,3.5){$g_2$}
         \put(4.40,1.05){$\pi$}
         \put(2.65,1.40){$O$}
         \put(2.00,2.05){$a_1$}
         \put(1.35,1.15){$a_2$}
         \put(3.30,1.15){$a_3$}
         \end{picture}
         {\refstepcounter{abbildung}\label{abb4}
          \centerline{Figure \ref{abb4}.}}
         \end{minipage}
      \end{center}
}%
\par
Next, we define a sexternary predicate $\equiv_{+}$, with $abc\,
\equiv_{+}  a'b'c'$ to be read as `$abc$ and $a'b'c'$ are either
both trilaterals or both tripods', by
\begin{eqnarray}\label{T}
 a_1b_1c_1 \equiv_{+}  a_2b_2c_2
 & :\Leftrightarrow &
 (\forall\, g)(\exists\, x_{11}x_{21}x_{12}x_{22}x_{13}x_{23})
 \nonumber\\
    & &
   \bigwedge_{i=1}^2
   \Bigg(T(a_ib_ic_i)
   \wedge \Big(\bigwedge_{j=1}^3 (x_{ij}\simeq a_i, b_i, c_i,g)
   \wedge (x_{ij}\neq x_{i,j+1})\Big)\Bigg)\\
   & & {}\wedge \Big(\bigwedge_{j=1}^3 x_{1j}\simeq x_{2j}\Big) .\nonumber
\end{eqnarray}
\par
Suppose that $a_1b_1c_1$ and $a_2b_2c_2$ are trilaterals in planes
$\pi_1$ and $\pi_2$, respectively. Then the lines $x_{ij}$ can be
chosen as follows: If (i) $\pi_1\neq\pi_2$ and if $g$ is skew to
the line $s=\pi_1\cap\pi_2$, then we choose three distinct points
$X_1, X_2, X_3$ on $s$, and we let $x_{ij}$ be the line joining
$X_j$ with $g\cap \pi_i$ (figure \ref{abb5}). If (ii) $\pi_1\neq
\pi_2$ and if $g$ and $s$ are not skew, then we choose $G$ to be a
point lying on both $g$ and $s$, and we let $x_{11}=x_{21}=s$, and
choose for $x_{i2}$ and $x_{i3}$ any two distinct lines through
$G$ in the plane $\pi_i$, which are different from $s$. If (iii)
$\pi_1=\pi_2=\pi$, then we let $x_{11}=x_{21}$, $x_{12}=x_{22}$,
and $x_{13}=x_{23}$ be any three distinct lines in $\pi$ through a
point common to $\pi$ and $g$. In case both $a_1b_1c_1$ and
$a_2b_2c_2$ are tripods, the reasoning is, by dint of duality,
similar.
\par
Should $a_1b_1c_1$ be a trilateral in a plane $\pi$, and
$a_2b_2c_2$ be a tripod with the vertex (point of concurrence)
$P$, then we let $g$ be a line which neither passes through $P$
nor lies in $\pi$ (figure \ref{abb6}). Let $G$ be the point of
intersection of $g$ with $\pi$, and let $\gamma$ be the plane
spanned  by $g$ and $P$. If lines $x_{ij}$ were to satisfy the
conditions in the second line of (\ref{T}), then $G\in
x_{1j}\subset \pi$ and $P\in x_{2j}\subset \gamma$, and since at
least two of the lines $x_{1j}$, say $x_{11}$ and $x_{12}$, must
be different from $\pi\cap \gamma$, the conditions $x_{11}\simeq
x_{21}$ and $x_{12}\simeq x_{22}$ imply that both $x_{21}$ and
$x_{22}$ have to be the line joining $P$ with $G$, so they cannot
be different, as required by the definiens in (\ref{T}).
\par
{\unitlength1.0cm
      \begin{center}
      \begin{minipage}[t]{6.0cm}
         \begin{picture}(6.0,3.5)
         \put(0.0 ,0.0){\includegraphics[width=6.0cm]{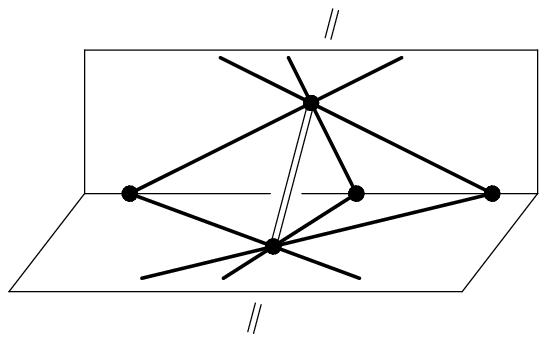}}
         \put(4.80,0.65){$\pi_1$}
         \put(5.55,2.9){$\pi_2$}
         \put(3.35,3.45){$g$}
         \put(1.05,1.8){$X_1$}
         \put(4.00,1.8){$X_2$}
         \put(5.40,1.8){$X_3$}
         \put(3.90,0.78){$x_{11}$}
         \put(2.76,0.60){$x_{12}$}
         \put(1.10,0.78){$x_{13}$}
         \put(4.50,2.85){$x_{21}$}
         \put(3.30,2.95){$x_{22}$}
         \put(2.00,2.85){$x_{23}$}

      \end{picture}
      {\refstepcounter{abbildung}\label{abb5}
       \centerline{Figure \ref{abb5}.}}
      \end{minipage}
      \hspace{1.5cm}
      \begin{minipage}[t]{6.0cm}
         \begin{picture}(6.0,3.5)
         \put(0.0 ,0.1){\includegraphics[width=6.0cm]{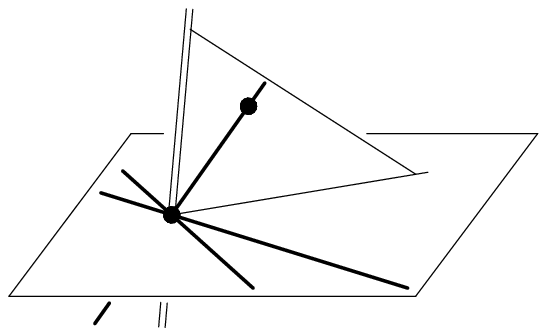}}
         \put(5.50,2.00){$\pi$}
         \put(1.50,0.95){$G$}
         \put(2.60,2.05){$P$}
         \put(1.70,3.5){$g$}
         \put(3.9 ,1.9){$\gamma$}
         \put(3.95,0.78){$x_{11}$}
         \put(2.8,0.60){$x_{12}$}
         \end{picture}
         {\refstepcounter{abbildung}\label{abb6}
          \centerline{Figure \ref{abb6}.}}
         \end{minipage}
      \end{center}
}%

\par
We now define the sexternary predicate $\equiv_{-}$, with $abc
\equiv_{-} a'b'c'$ standing for `$abc$ and $a'b'c'$ are (in any
order) a trilateral and a tripod', by
\begin{eqnarray} \label{N}
a_1b_1c_1 \equiv_{-}  a_2b_2c_2 & :\Leftrightarrow & (\forall\,
g)(\exists\, x_1x_2)\, \bigwedge_{i=1}^2 \Big((x_i\simeq a_i, b_i,
c_i)\wedge
T(a_ib_ic_i)\Big)\\
& & {} \wedge \Big(\bigvee_{i=1}^2 (g=x_i\vee a_ib_ic_i\equiv_{+}
gx_1x_2 )\Big).\nonumber
\end{eqnarray}

\par
Suppose $a_1b_1c_1$ is a trilateral, lying in the plane $\pi$, and
$a_2b_2c_2$ is a tripod, with vertex $P$. If $g$ is a line in
$\pi$ then we choose $x_1=g$ and as $x_2$ any line through $P$.
The case that $g$ passes through $P$ can be treated similarly.
Hence we may restrict our attention to the case in which $g$
neither lies in $\pi$ nor passes through $P$, and denote in this
case by $G$ the point of intersection of $g$ and $\pi$, and by
$\gamma$ the plane spanned by $P$ and $g$.
\par
Then (i) if $P\not\in \pi$, we let $x_2$ be the line joining $P$
and $G$, and $x_1$ be any line in $\pi$ passing through $G$ and
different from the line $\pi\cap\gamma$ (figure \ref{abb7}), and
(ii) if $P\in \pi$, then we let $x_2$ be the line joining $P$ with
$G$, and we let $x_1$ be any line in $\pi$ passing through $G$,
but different from $x_2$ (figure \ref{abb8}).
\par
Now if both $a_1b_1c_1$ and $a_2b_2c_2$ were trilaterals lying in
the same plane, then  for any line $g$ not lying in that plane, we
could not find $x_1$ and $x_2$ with the desired properties, as the
requirement that $\bigwedge_{i=1}^2 (x_i\simeq a_i, b_i, c_i)$
forces them to lie in $\pi$, and so they can neither be equal to
$g$ nor form a trilateral with it. If both $a_1b_1c_1$ and
$a_2b_2c_2$ were trilaterals lying in different planes $\pi_1$ and
$\pi_2$, whose line of intersection is $l$, then for any line $g$
intersecting $l$ but lying neither in $\pi_1$ nor in $\pi_2$, we
could not find the desired $x_1$ and $x_2$, as the condition
$\bigwedge_{i=1}^2 (x_i\simeq a_i, b_i, c_i)$ forces them to lie
in $\pi_1$ and $\pi_2$, so they can neither be equal to $g$, nor
from a trilateral with it. A dual reasoning shows that, if
$a_1b_1c_1$ and $a_2b_2c_2$ were both tripods, (\ref{N}) could not
hold.
\par
{\unitlength1.0cm
      \begin{center}
      \begin{minipage}[t]{6.0cm}
         \begin{picture}(6.0,3.7)
         \put(0.0 ,0.1){\includegraphics[width=6.0cm]{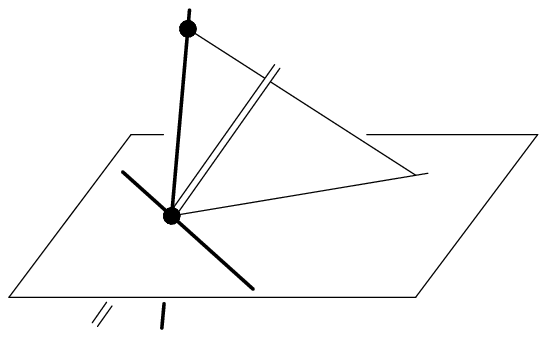}}
         \put(4.40,0.55){$\pi$}
         \put(1.50,0.95){$G$}
         \put(2.60,2.05){$g$}
         \put(1.6,3.4){$P$}
         \put(3.9 ,1.9){$\gamma$}
         \put(2.8,0.60){$x_{1}$}
         \put(1.5,2.50){$x_{2}$}
      \end{picture}
      {\refstepcounter{abbildung}\label{abb7}
       \centerline{Figure \ref{abb7}.}}
      \end{minipage}
      \hspace{1.5cm}
      \begin{minipage}[t]{6.0cm}
         \begin{picture}(6.0,3.7)
         \put(0.0 ,0.0){\includegraphics[width=6.0cm]{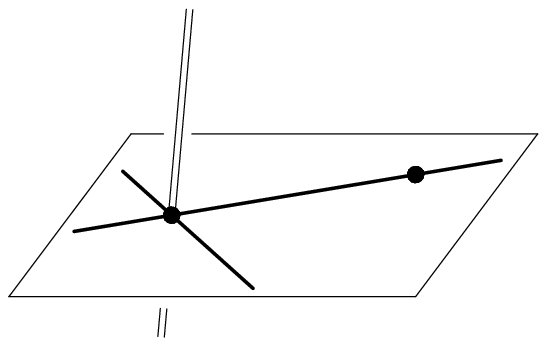}}
         \put(4.40,0.6){$\pi$}
         \put(1.50,0.95){$G$}
         \put(4.50,1.40){$P$}
         \put(1.70,3.5){$g$}
         \put(3.1,1.30){$x_{2}$}
         \put(2.8,0.65){$x_{1}$}

         \end{picture}
         {\refstepcounter{abbildung}\label{abb8}
          \centerline{Figure \ref{abb8}.}}
         \end{minipage}
      \end{center}
}%

\par
The sexternary predicate $\equiv_{\oplus}$, with $abc
\equiv_{\oplus}  a'b'c'$ standing for `$abc$ and $a'b'c'$ are both
trilaterals lying in \emph{different} planes or both tripods with
\emph{different} vertices', is defined by
\begin{eqnarray} \label{U}
a_1b_1c_1 \equiv_{\oplus}  a_2b_2c_2 & :\Leftrightarrow &
(\exists\, x_1x_2x_3)\, a_1b_1c_1 \equiv_{+}  a_2b_2c_2 \wedge
(x_3\sim
a_1,b_1,c_1,a_2, b_2,c_2)\\
& & {} \wedge a_1b_1c_1 \equiv_{-} x_1x_2x_3 \wedge
\bigg(\bigwedge_{i=1}^2 (x_i\sim a_i,b_i,c_i)\bigg).\nonumber
\end{eqnarray}
If $a_1b_1c_1$ and $a_2b_2c_2$ are both trilaterals (the tripod
case is treated dually), lying in different planes $\pi_1$ and
$\pi_2$ intersecting in $g$, then we choose a point $P$ on $g$ as
the vertex of a tripod $x_1x_2x_3$, where $x_3=g$, $x_1$ lies in
$\pi_1$, and $x_2$ lies in $\pi_2$. If $a_1b_1c_1$ and $a_2b_2c_2$
were both trilaterals lying in the same plane $\pi$, then any
$x_1,x_2,x_3$ satisfying the intersection conditions of (\ref{U})
would have to belong to $\pi$, and thus could not form a tripod.
\par
We are finally ready to define positively, with $\neq$ allowed,
the skewness predicate $\sigma$, with $\sigma(ab)$ to be read `the
lines $a$ and $b$ are skew', by
\begin{eqnarray} \label{neq}
\sigma(ab) & :\Leftrightarrow & (\forall\, g)(\exists\,
xa_1a_2b_1b_2)\,
(x\sim a,b) \wedge (x\simeq g) \\
& & \bigwedge_{i=1}^2 (aa_ix \equiv_{+}  bb_ix \wedge
aa_ix\equiv_{\oplus}  bb_ix )\wedge aa_1x \equiv_{-}
aa_2x.\nonumber
\end{eqnarray}
Suppose $a$ and $b$ are skew, and let $P$ be a point on $a$
(figure \ref{abb10}). The line $g$ must have a point $R$ in common
with the plane determined by $P$ and $b$. Let $x$ be a line
containing $P$, $R$ and intersecting $b$ in a point $Q$. Let $a_1$
be any line through $P$ that does not lie in plane determined by
$a$ and $x$, $a_2$ be any line intersecting both $x$ and $a$ in
points different from $P$, $b_1$ a line through $Q$  not in the
plane determined by $b$ and $x$, and $b_2$ a line intersecting $b$
and $x$ in points different from $Q$. With these choices the
definiens in (\ref{neq}) is satisfied.
\par
Should $a$ intersect $b$, and should $g$ be chosen such that $abg$
forms a tripod with vertex $P$, then, given that $(x\sim
a,b)\wedge (x\simeq g)$, the $x$ required to exist by (\ref{neq})
would have to pass through $P$. Since $aa_1x \equiv_{-} aa_2x$,
one of $aa_1x$ or $aa_2x$ must be a tripod. W. l.\ o.\ g.\ we may
suppose $aa_1x$ is a tripod. By $aa_1x\equiv_{+}   bb_1x$, $bb_1x$
must be a tripod as well, and by $aa_1x \equiv_{\oplus} bb_1x$ the
two tripods must have different vertices, which is impossible,
for, regardless of the choice of $a_1$ and $b_1$, the vertex of
both tripods is $P$.
\par
{\unitlength1.0cm
      \begin{center}
      \begin{minipage}[t]{6.0cm}
         \begin{picture}(6.0,4.3)
         \put(0.0 ,0.0){\includegraphics[width=6.0cm]{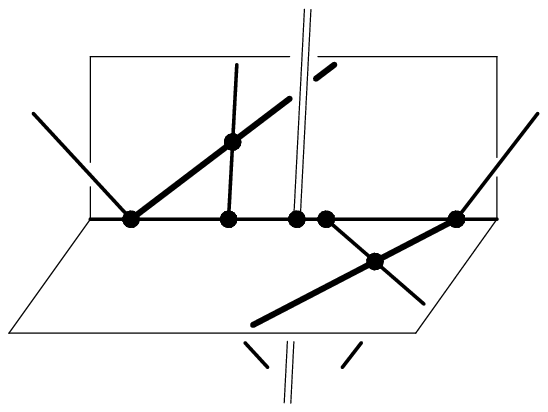}}
         \put(2.55,1.00){$a$}
         \put(3.75,3.50){$b$}
         \put(3.10,4.20){$g$}
         \put(4.00,2.20){$x$}
         \put(5.65,3.45){$a_1$}
         \put(4.25,0.95){$a_2$}
         \put(0.35,3.3){$b_1$}
         \put(2.20,3.50){$b_2$}
         \put(4.90,1.65){$P$}
         \put(1.25,1.65){$Q$}
         \put(3.10,1.65){$R$}
         \end{picture}
         {\refstepcounter{abbildung}\label{abb10}
          \centerline{Figure \ref{abb10}.}}
         \end{minipage}
      \end{center}
}%
\par
The positive definition (in terms of $\sim$ with $\neq$ allowed)
of the non-intersection predicate we were looking for in the
three-dimensional case is
\begin{equation}\label{3-neq}
a\not\sim b\;:\Leftrightarrow\; a=b \vee \sigma(ab).
\end{equation}
However, we do not know whether $\neq$, the negated line equality,
is positively definable in terms of $\sim$, and thus whether it is
possible to have a thoroughly positive definiens in (\ref{3-neq}).

\section{Affine spaces}

Notice that (\ref{1})--(\ref{4}) are valid in $n$-dimensional
affine geometry with $n\geq 3$ as well, since for any plane there
is a disjoint parallel line.
\par
Since (\ref{4}) holds, any surjective map between the sets of
lines of two affine spaces of dimension $n\geq 3$, which maps
intersecting lines into intersecting lines must be injective as
well.
\par
In affine geometry, we distinguish two cases: (A) the one in which
every line is incident with exactly two points (and then the space
can be coordinatized by $\GF(2)$), and (B) the one in which every
line is incident with at least three points.  The number of all
lines is $k:=2^{n-1}(2^n-1)$ in case (A), whereas in case (B) this
number is strictly greater than $k$. Hence we can characterize
cases (A) and (B) by
\begin{eqnarray}\label{(A)}
\alpha \;:\Leftrightarrow\; (\forall\, x_1\ldots x_{k+1})\,
(\bigvee_{1\leq i<j\leq k+1} x_i=x_j)
\end{eqnarray}
and $\neg\,\alpha$, respectively. It is worth noticing that the
negated equalities in $\neg\alpha$ can be avoided altogether,
without using (\ref{4}), and that the number of variables in
$\neg\alpha$ can be greatly reduced, by taking into account that
in case (A) there are no more than $2^n-1$ pairwise intersecting
lines, namely all the lines through a fixed point, whereas in case
(B) this number is exceeded. Therefore
\begin{eqnarray}\label{(B)}
\beta \;:\Leftrightarrow\; (\exists\, x_1\ldots x_{2^n})\,
(\bigvee_{1\leq i<j\leq 2^n} x_i \sim x_j)
\end{eqnarray}
positively characterizes case (B).
\par
Affine geometry can be axiomatized in terms of points and lines,
with point-line incidence and line-parallelism as primitive
notions, and the first such axiomatization was presented in
\cite[\S2]{len}. Affine geometry of a fixed dimension $n\geq 3$,
in which (A) holds, cannot be axiomatized inside ${\cal L}$, as it
is not possible to define the line-parallelism predicate
$\parallel$ in terms of line-intersection, given that there are
maps that preserve both $\sim$ and $\not\sim$, but which do not
preserve $\parallel$, but it can be axiomatized in terms of lines,
$\sim$, and $\parallel$. Affine geometry of a fixed dimension
$n\geq 3$, in which (B) holds, can be axiomatized inside ${\cal
L}$, by rephrasing the axiom system in \cite[\S2]{len} in terms of
lines and $\sim$ (this is possible in this case as $a\parallel b$
can be replaced by $\pi(ab)\wedge a\not\sim b$, where $\pi$ is the
coplanarity predicate defined below in (\ref{coplanar})), and by
adding suitable dimension axioms. However, regardless of whether
(A) or its negation has been added to the axiom system of
$n$-dimensional affine geometry with $n\geq 3$, it is true that
$\not\sim$ can be defined positively  in terms of $\sim$, given
that $\neq$, which occurs in (\ref{13}),  can be defined
positively by means of (\ref{4}).

\par
If every line contains exactly two points, i.~e.\ in case (A),
then it is quite easy to define positively the non-intersection
predicate by observing that, if two different lines do not
intersect, then there is more than one line that intersects the
two lines in different points, but if they do intersect there is
only one such line. Therefore the definition in this case is
\begin{equation}\label{13}
a_1\not\sim a_2 \;:\Leftrightarrow\; a_1=a_2\vee \Bigg(\alpha
\wedge (\exists\, b_1b_{2})\, b_1\neq b_2 \wedge\bigg(
\bigwedge_{i=1}^{p}a_1b_i\,\#\,a_2 b_i\bigg)\Bigg).
\end{equation}
We denote the definiens of this definition by $\gamma$. The
conjunct $\alpha$ in (\ref{13}) is not needed if we regard it
plainly as a definition of non-intersection inside the ${\cal
L}$-theory of $n$-dimensional affine spaces over $\GF(2)$, but we
shall use $\gamma$ in the general case, where we have no
information regarding the number of points incident with a line,
below, and there we do need that conjunct as well.
\par
From now on, we assume that lines are incident with more than two
points. For all dimensions $n\geq 3$ we can define the coplanarity
$\pi$ of two lines (which are allowed to coincide) by
\begin{equation}\label{coplanar}
\pi(ab)\;:\Leftrightarrow\; (\exists\, cde) S(acd)\wedge
S(bce)\wedge d\sim b \wedge d\sim e \wedge e\sim a.
\end{equation}
See figure \ref{abb11}.
\par
{\unitlength1.0cm
      \begin{center}
      \begin{minipage}[t]{6.0cm}
         \begin{picture}(6.0,3.2)
         \put(0.0 ,0.1){\includegraphics[width=6.0cm]{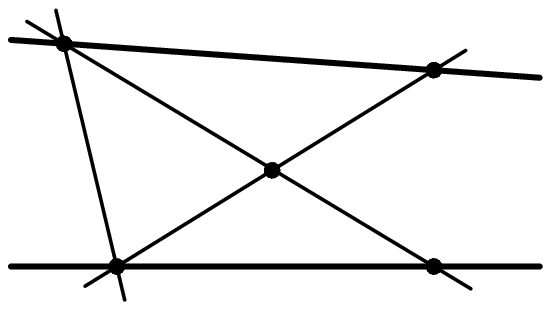}}
         \put(5.5,2.8){$a$}
         \put(5.5,0.1){$b$}
         \put(0.65,1.5){$c$}
         \put(4.0,1.9){$e$}
         \put(4.0,1.05){$d$}
      \end{picture}
      {\refstepcounter{abbildung}\label{abb11}
       \centerline{Figure \ref{abb11}.}}
      \end{minipage}
      \hspace{1.5cm}
      \begin{minipage}[t]{6.0cm}
         \begin{picture}(6.0,3.2)
         \put(0.0 ,0.0){\includegraphics[width=6.0cm]{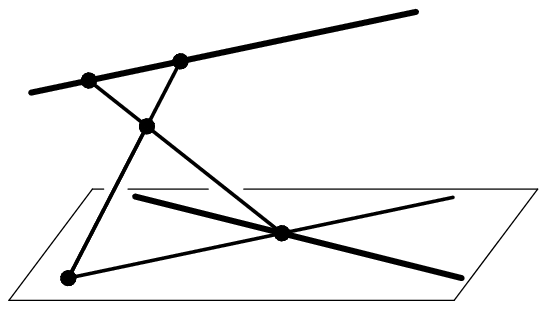}}
         \put(3.35,0.15){$V=a_1$}
         \put(4.15,2.8){$\overline{x}=a_2$}
         \put(4.45,0.75){$x\parallel a_2$}
         \put(0.25,0.08){$P$}
         \put(3.0,0.92){$X$}
         \put(0.80,2.7){$Y$}
         \put(1.10,1.7){$Z$}
         \put(1.85,2.15){$b_2$}
         \put(2.3,1.5){$b_1$}
         \end{picture}
         {\refstepcounter{abbildung}\label{abb12}
          \centerline{Figure \ref{abb12}.}}
         \end{minipage}
      \end{center}
}%
\par
To define non-intersection in $n$-dimensional affine space with
$n\geq 3$, we need the following
\begin{lem}
Let $n\geq 3$, $m=[\frac{n+1}{2}]$, let $a_1,\ldots, a_m$ be $m$
independent lines in $n$-dimensional affine space, let $U=\langle
a_1, \ldots, a_m\rangle$ be the subspace spanned by these lines,
and let $V=\langle a_1, \ldots, a_{m-1}\rangle$. Then for any
point $P\in U$ there are (not necessarily distinct) lines $b_1$
and $b_2$, such that $b_1$ joins a point in $V$ with a point on
$a_m$, $b_2$ joins a point in $V$ or in $a_m$ with a different
point on $b_1$, and $P$ lies on $b_2$.
\end{lem}

\begin{proof}
If $P$ is on $a_m$ (or if $P\in V$), then choose $b_1=b_2$ to be a
line joining $P$ with a point  in $V$ (or in $a_m$). If $P$ is
neither on $a_m$ nor in $V$, then the subspaces $\langle P,
a_m\rangle$ and $\langle P, V\rangle$ intersect in a line $x$. If
$x$ intersects both $a_m$ and $V$ in a point, then we let
$b_1=b_2=x$. Since $x$ cannot be parallel to both $x$ and $V$, if
it doesn't intersect both, it may be parallel to only one of them,
i.~e.\ either (i) $x\parallel V$ or (ii) $x\parallel  a_m$. Let
$X$ be the point of intersection of $x$ with (i) $a_m$ or (ii)
$V$. Let $Y$ be a point in (i) $V$ or (ii) $a_m$, let
$\overline{x}$ be the parallel through $Y$ to $x$, and
$b_1:=\langle X, Y \rangle$. (Figure \ref{abb12} depicts case (ii)
for $m=2$, so that $V=a_1$ and $\overline{x}=a_2$.) Let $Z$ be a
third point on $b_1$ and let $b_2:=\langle P, Z\rangle$. The line
$b_2$ is not parallel to $\overline{x}$ and thus intersects (i)
$V$ or (ii) $a_m$ in a point which is different from $Z$.
\end{proof}

We now define some auxiliary predicates. Let $M(a_1\ldots a_mx)$
stand for `$x$ is one of the lines $a_i$ or it intersects two of
these lines in different points', i.~e.
\begin{equation}
M(a_1\ldots a_mx)\;:\Leftrightarrow\; \Big(\bigvee_{i=1}^m
x=a_i\Big)\vee \Big(\bigvee_{1\leq i<j\leq m} a_ix\,\#\,
a_jx\Big).
\end{equation}
Closely related to $M$, we introduce
\begin{equation}\label{Mq}
M_q(a_1\ldots a_mx)\;: \Leftrightarrow\;(\exists\, b_1\ldots
b_q)\, \bigwedge_{i=1}^q M(a_1\ldots a_mb_1\ldots b_i) \wedge
M(a_1\ldots a_m b_1\ldots b_q x).
\end{equation}
If (\ref{Mq}) holds then the line $x$ belongs to the affine
subspace spanned by $a_1,...,a_m$, since it can be `reached' with
the help of the auxiliary lines $b_1,\ldots,b_q$.

\par
With $m$ standing for  $[\frac{n+1}{2}]$,  whenever $a_1\not\sim
a_2$, we can find lines $a_3, \ldots, a_m$ such that $a_1, \ldots,
a_m$ are independent. Let $U$ be the subspace spanned by them. We
infer from the above lemma, that each line $h$ in $U$ satisfies
$M_r(a_1\ldots a_mh)$ for $r=2^{m+1}-4$. Recall that $\beta$
ensures that we are in case (B). So we can now state the
definition of non-intersection, when $n$ is even (in this case $U$
is a hyperplane, so that to any line $g$ there exists a line $h$
in $U$ coplanar with $g$) as
\begin{equation}\label{delta0}
a_1\not\sim a_2: \;\Leftrightarrow\; a_1=a_2 \vee \Big(\beta\wedge
 (\exists\,
a_3\ldots a_m)(\forall\, g)(\exists\, h)\, \pi(gh)\wedge
M_r(a_1\ldots a_mh)\Big).
\end{equation}
If $n$ is odd, $U$ is the whole affine space, so any line $g$ lies
in $U$, and thus
\begin{equation}\label{delta1}
a_1\not\sim a_2: \;\Leftrightarrow\; a_1=a_2 \vee \Big(\beta\wedge
(\exists\, a_3\ldots a_m)(\forall\, g)\, M_r(a_1\ldots a_mg)\Big).
\end{equation}
The definiens of the definitions in (\ref{delta0}) and
(\ref{delta1}) are denoted by $\delta_0$ and $\delta_1$,
respectively.
\par
Finally, we return to the general case of $n$-dimensional affine
geometry. By (\ref{13}), (\ref{delta0}), and (\ref{delta1}) the
definition of non-intersection is
\begin{equation}\label{allgemein}
a_1\not\sim a_2\;:\Leftrightarrow\; \gamma \vee
\delta_{2\left(\frac{n}{2}-[\frac{n}{2}]\right)}.
\end{equation}

\section{Higher-dimensional subspaces}

Given \cite{cho}, $n$-dimensional projective geometry can also be
axiomatized with $k$-dimensional subspaces (for all $1\leq k\leq
n-1$ with $2k+1\neq n$) as individual variables and a binary
intersection predicate $\sim$, with $a\sim b$ to be interpreted as
`the subspaces $a$ and $b$ intersect in a $k-1$-dimensional
subspace'. From the results in \cite{hua} it follows that the
non-intersection predicate is also positively definable in terms
of the intersection predicate (negated equality is allowed), but
the actual definition will very likely be prohibitively intricate.

\small

\par
\noindent Hans Havlicek, Institut f\"ur Geometrie, Technische
Universit\"at Wien, Wiedner Hauptstra{\ss}e 8--10/1133, \linebreak
A--1040 Wien, Austria.
\par
\noindent Victor Pambuccian, Department of Integrative Studies,
Arizona State University West, P.~O.~Box 37100, Phoenix, AZ
85069-7100, U.S.A.

\end{document}